\documentclass[11pt]{amsart}
\usepackage{amssymb}
\usepackage{amsfonts}
\usepackage{amsmath}
\usepackage{graphicx}
\usepackage{xcolor}
\usepackage{amsmath}
\usepackage{mathrsfs}
\usepackage{stmaryrd}
\usepackage{epsfig,color}
\usepackage{blindtext}
\usepackage{enumerate}
\usepackage[citecolor=blue,colorlinks]{hyperref}

\usepackage{url}
\usepackage{bbm}
\usepackage{nicefrac,mathtools}
\usepackage{bm}   
\DeclareGraphicsExtensions{.pdf,.jpeg,.png}
\usepackage{epstopdf}
\usepackage{cancel} 
\usepackage[normalem]{ulem} 
\usepackage{verbatim} 
\usepackage{enumitem} 
\usepackage{color}
\usepackage[msc-links, lite]{amsrefs}
\usepackage{geometry}
\newtheorem{theorem}{Theorem}[section]
\newtheorem{conjecture}[theorem]{Conjecture}
\newtheorem{proposition}[theorem]{Proposition}
\newtheorem{lemma}[theorem]{Lemma}
\newtheorem{corollary}[theorem]{Corollary}

\theoremstyle{definition}
\newtheorem{definition}[theorem]{Definition}
\newtheorem{remark}[theorem]{Remark}

\numberwithin{equation}{section}

\newcommand{\mb}{\mathbb}
\newcommand{\mc}{\mathcal}

\newcommand{\Sc}{\mathrm{Sc}}
\newcommand{\mr}{\mathrm}

\DeclareMathOperator{\Ric}{Ric}

\newcommand{\hess}{{\mathrm{Hess}}}
\newcommand{\vol}{\mathrm{vol}}

\def\Xint#1{\mathchoice
{\XXint\displaystyle\textstyle{#1}}%
{\XXint\textstyle\scriptstyle{#1}}%
{\XXint\scriptstyle\scriptscriptstyle{#1}}%
{\XXint\scriptscriptstyle\scriptscriptstyle{#1}}%
\!\int}
\def\XXint#1#2#3{{\setbox0=\hbox{$#1{#2#3}{\int}$ }
\vcenter{\hbox{$#2#3$ }}\kern-.6\wd0}}

\def\dashint{\Xint-}

\newcommand{\N}{\mathbb{N}}

\newcommand{\R}{\mathbb{R}}
\renewcommand{\subset}{\subseteq}
\newcommand{\defeq}{\mathrel{\mathop:}=}

\newcommand{\haus}{\mathcal{H}}
\newcommand{\leb}{\mathcal{L}}

\newcommand{\dist}{\mathsf{d}}

\newcommand{\meas}{\mathfrak{m}}

\DeclareMathOperator{\RCD}{RCD}

\usepackage{todonotes}

\title{Positive Scalar Curvature Meets Ricci Limit Spaces}
\date{\today}

\author{Jinmin Wang}\address{Department of Mathematics, Texas A\&M University}\email{jinmin@tamu.edu}

\author{ZhiZhang Xie}\address{Department of Mathematics, Texas A\&M University}\email{xie@tamu.edu}

\author{Bo Zhu}\address{Department of Mathematics, Texas A\&M University}\email{bozhu@tamu.edu}

\author{Xingyu Zhu}\address{Institute for Applied Mathematics, University of Bonn }\email{zhu@iam.uni-bonn.de}

\begin{document}

\begin{abstract}
We investigate the influence of uniformly positive scalar curvature on the size of a non-collapsed Ricci limit space coming from a sequence of $n$-manifolds with non-negative Ricci curvature and uniformly positive scalar curvature. We prove that such a limit space splits at most $n-2$ lines or $\R$-factors. When this maximal splitting occurs, we obtain a uniform upper bound on the diameter of the non-splitting factor. Moreover, we obtain a volume gap estimate and a volume growth order estimate of geodesic balls on such manifolds.

\bigskip


\end{abstract}

\maketitle

\section{Introduction}
In this article, we focus on the topological and geometric implications of uniformly positive scalar curvature following the philosophy of Gromov's macroscopic dimension conjecture. The conjecture essentially posits that any $n$-dimensional Riemannian manifold with uniformly positive scalar curvature is at most $(n-2)$-dimensional at large scales. Our particular interest lies in establishing an analogous statement or conjecture in the context of Ricci limit spaces.

Suppose that $(X,\dist, x)$ is a non-collapsed Ricci limit space that is obtained from a pointed Gromov--Hausdorff (pGH in short) limit of a sequence of Riemannian manifolds with non-negative Ricci curvature and uniformly positive scalar curvature, the large scale geometry can be understood as the number of $\R$ factor that split off from $X$. Motivated by the Gromov's macroscopic dimension conjecture, in our context, we obtain that for $X$ as described above, the number of $\R$-factors can be split off is at most $n-2$, and if the maximal splitting happens, i.e., if $X=\R^{n-2}\times Y$ isometrically, then certain topological and geometric constraints is imposed upon on $Y$, which heuristically inherit from the uniformly positive scalar curvature. The challenge is that a proper notion of scalar curvature is not yet defined on Ricci limit spaces, which remains an open question in the community of studies on the scalar curvature. Here, with a careful investigation on the metric structure of manifolds rather than introducing a notion of scalar curvature on limit spaces, we prove that,

\begin{theorem}\label{thm:limit_n-2}
If $(X,\dist, p)$ is a (pointed) non-collapsed Ricci limit space that is a {\rm pGH} limit of a sequence of $n$-dimensional Riemannian manifolds $(M_i, g_i, p_i)$ with $\Ric_{g_i}\ge 0$, $\mr{Sc}_{g_i}\ge 2$ and $\vol_{g_i}(B(p_i,1))>v>0$ for every $i\in \N$, then
\begin{enumerate}
    \item\label{main:item1} $X$ cannot isometrically split off $\R^{n-1}$.
    \item\label{main:item2} If $X$ splits off $\R^{n-2}$ and we write  $X=\R^{n-2}\times Y$, then $Y$ is isometric to either $(\mb{S}^2,\dist_{\mb{S}^2})$ or $(\R P^2,\dist_{\R P^2})$ with a metric of non-negative curvature in the sense of Alexandrov geometry. 
    \item\label{main:item3} If $X$ is isometric to $\R^{n-2}\times \mb S^2$, and either of the following two conditions is satisfied
    \begin{enumerate}
        \item The $\mb S^2$ has no singular points, i.e., $\mc{R}(\mb S^2)=\mb S^2$;
        \item $3\le n\le 8$,       
    \end{enumerate}
    then
    \begin{equation}\label{eq:diambd}
        \mr{diam}(\mb{S}^2, \dist_{\mb{S}^2})\le \sqrt{\frac{2 (n-1)}{n}}\pi.
    \end{equation}
    Here, $\mc{R}(S^2)$ denotes the regular set of $\mb S^2$ (see section \ref{sec:prelim} for the details).
\end{enumerate}
\end{theorem}

\begin{remark}\label{rem:codim4}
    The condition $\mc{R}(\mb S^2)=\mb S^2$ is not the weakest possible condition to guarantee the validity of our theorem but it is a natural one. For example it is satisfied when there is a Ricci curvature upper bound by the codimension 4 theorem of Cheeger--Naber \cite{CN15_codim4}. Also note that the metrics satisfying $\mc R(\mb S^2)=\mb S^2$ are dense in all metrics that support a non-collapsed $\RCD(0,2)$ structure in GH topology, because every $\mb S^2$ of nonnegative curvature in the sense of Alexandrov geometry can be approximated by $\mb S^2$ of nonnegative sectional curvature, see for example \cite{RichardSmoothing}.
\end{remark}

The diameter upper bound established in \eqref{eq:diambd} is not sharp. We describe the intuition to the expected sharp bound. Given the assumptions in Theorem \ref{thm:limit_n-2} and the maximal splitting $X=\R^{n-2}\times Y$, if we anticipate that the scalar curvature lower bound is preserved under GH convergence, then the flat $\R^{n-2}$ factor in $X$ should not inherit any positivity of scalar curvature. Consequently, the positive scalar curvature would be carried fully onto the $\mathbb{S}^2$-factor. In such a scenario, the scalar curvature lower bound on $\mathbb{S}^2$ should be equivalent to the sectional curvature lower bound in the sense of Alexandrov geometry. This will lead to a Bonnet--Myer type diameter upper bound. An analogous characterization on the $2$-area has been proved in \cite{ZhuJT_area}. Here, we propose a conjecture characterizing the behavior of uniformly positive scalar curvature under {\rm pGH} convergence.

\begin{conjecture}\label{thm: Bonnet-Myer-type}
In the setting of Theorem \ref{thm:limit_n-2}, if $X=\R^{n-2}\times \mb S^2$, then 
\begin{equation}
    \mr{diam}(\mb{S}^2, \dist_{\mb{S}^2})\le \pi.
\end{equation}
Moreover, the equality holds if and only if $(\mb S^2,\dist_{\mb{S}^2})$ is isometric to the spherical suspension of $\mb S^1$ with $\mr{diam}(\mb S^1)\leq \pi$.
\end{conjecture}
Note that the third and fourth named authors \cite{Zhu-Zhu_optimal} confirm Conjecture \ref{thm: Bonnet-Myer-type} for $n=3$. Moreover, the virtue of the proof of Theorem \ref{thm:limit_n-2} also enables us to derive an upper bound on the first Betti number, and this upper bound can be viewed as a topological consequence of having uniformly positive scalar curvature in a Ricci limit space. 

\begin{corollary}\label{cor:betti}
In the setting of Theorem \ref{thm:limit_n-2}, the first Betti number of $X$, denoted by $b_1(X)$, satisfies
\begin{equation}\label{eq:bettiRLS}
    b_1(X)\le n-2.
\end{equation}    

\end{corollary}

\begin{remark}
    Note that for non-compact manifolds $(M^n,g)$ with non-negative Ricci curvature, we have $b_1(M) \leq n-1$. See also a version of this result for $\RCD$ spaces in \cite{Ye_bettiRCD}. Inspired by Gromov's macroscopic dimension conjecture, we expect a sharp upper bound $b_1(M) \leq n-3$ for any non-compact manifold $(M,g)$ with non-negative Ricci curvature and uniformly positive scalar curvature. However, the upper bound $n-3$ remains open if $n\ge 4$. We also point out that the sharp upper bound $n-3$ can be derived from \cite{andersontopology} and the volume growth conjecture raised by Gromov in \cite{Zhu_Geometryofpsc}*{Problem 1.6}. A weaker version of this conjecture will be obtained in Theorem \ref{thm:splitting} below.
\end{remark}
Before advancing to the next result, we discuss some related work here. It is known that the scalar curvature lower bound along is not stable under Gromov-Hausdorff convergence even we assume that the limit space is smooth (see \cite{Lee_Topping_example}). To continue the study of stability of scalar curvature lower bound, one either consider other notions of convergence, for instance \cites{Gromov_DP, Sormani_IFC}, or impose additional conditions (or do both, see \cite{LNN21}). A natural additional condition is a stronger curvature lower bound.  
It is still an open question whether the scalar curvature lower bound is stable under the Gromov-Hausdorff convergence with additional curvature assumptions. It is proved in \cite{petrunin_curvature_tensor} (after \cite{Petrunin_ScalarIntegral}) that if a noncollapsing sequence of manifolds $(M_i, g_i)$ has a uniform sectional curvature lower bound, then one can pass $\Sc_{g_i} \cdot \vol_{g_i}$ to the GH limit as a Radon measure, which can be regarded as some generalized scalar curvature, but the geometric meaning of it is yet to be revealed. Moreover, if sectional curvature lower bound is weakened to a Ricci curvature one, Naber conjectures that the same measure convergence holds and describes some potential applications in \cite{naber_open_question}. Along this line of thoughts, Theorem \ref{thm:limit_n-2} provides some evidences for the stability of scalar curvature under additional condition of nonnegative Ricci curvature. Some other study of stability of scalar curvature lower bounds under extra conditions includes \cites{Robin_scalar, KKL22}. From a technical point of view, a similar use of cube inequality (Theorem \ref{thm:cubeineq}) we make here can also be found in \cites{Dongpsc, Zhu_Geometryofpsc}.  



\vspace{1cm}
 Next, we study the effect of positive scalar curvature on the volume of geodesic balls on complete Riemannian manifolds with non-negative Ricci curvature. Ricci limit spaces do not make its appearance in the statement but do play a role in the proof. To motivate our study, we recall that if $B(p,r)\subset M$ is a geodesic ball for small $r> 0$, then 
\[
    \vol_g(B(p,r))=\omega_n r^n\left(1-\frac{\Sc_g(p)}{6(n+2)}r^2+O(r^4)\right).
\]
Here, $\omega_n$ is the volume of the unit ball in $\R^n$. The expansion indicates that the more positive $\Sc_g(p)$ is, the smaller $\vol_g(B(p,r))$ is, but it is only valid for small radii. 
However, under the extra assumption of non-negative Ricci curvature, we can prove a volume gap theorem for a wider range of radii.

\begin{theorem} \label{thm:gap}
Suppose that $(M^n,g)$ is a complete, non-compact, $n$-dimensional Riemannian manifold with non-negative Ricci curvature $\Ric_{g} \geq 0$. If  $\Sc_{g} \geq 4\pi^2 n(n-1)$, then there exists a constant $c\defeq c(n)$ such that
\begin{equation}
    \sup_{p \in M}\mr{vol}_g(B(p,1)) \leq c < \omega_n.
\end{equation}
\end{theorem}

We also have a version of this theorem for $M$ with nonempty boundary, based on the boundary regularity theory of non-collapsed $\RCD(K,N)$ spaces. See section \ref{sec:RCD} for relevant definitions. 

\begin{corollary}\label{cor:gap_boundary}
Let $(M^n,g)$ be a complete, non-compact, $n$-dimensional manifold with convex boundary $\partial M$, non-negative Ricci curvature $\Ric_{g} \geq 0$ and uniformly positive scalar curvature $\Sc_{g} \geq 4\pi^2 n(n-1)$ in the interior $M\setminus \partial M$, then there exists a constant $c\defeq c(n)$ such that
\begin{equation}\label{eq:gap_bdry}
    \sup_{p \in \partial M}\mr{vol}_g(B(p,1)) \leq c < \frac12\omega_n.
\end{equation}
\end{corollary}

\begin{remark}
Corollary \ref{cor:gap_boundary} can also be obtained by taking the double of $M$. Note that the boundary of $M$ is convex, we can smooth the metric on double of $M$ near the boundary of $M$ to obtain a complete, non-compact manifold with non-negative Ricci curvature and uniformly positive scalar curvature. The details can be seen from \cite{Chow_Positivity} and the Appendix in \cite{Gromov_metric_inequality}.  
\end{remark}


\bigskip
When we restrict the consideration of Ricci limit spaces to the ones obtained at infinity of a manifold with non-negative Ricci curvature and uniformly positive scalar curvature, we can extract some information on the volume growth rate of geodesic balls. We prove, 

\begin{theorem}\label{thm:splitting}
If $(M^n,g)$ is a complete, non-compact, $n$-dimensional manifold with $\Ric_{g}\ge 0$ and $\Sc_{g}\ge n(n-1)$, then there exists a constant $c(n)>0$ such that
\begin{equation}\label{eq:infvol}
    \inf_{p\in M} \vol_g(B(p,r))\le c(n)r^{n-2}, \quad \forall r\ge 0.
\end{equation}
Moreover, the asymptotic volume ratio is $0$, i.e.,
\begin{equation}\label{eq:avr0}
    \lim_{r\to \infty} \frac{\vol_g(B(p,r))}{r^n}=0.
\end{equation}

\end{theorem}
 Theorem \ref{thm:splitting} is an attempt to partially answer Gromov's conjecture \cite{Zhu_Geometryofpsc}*{Problem 1.6}. 
It can be seen as a step to understand the asymptotic behavior of volume growth of the geodesic ball in a complete, non-compact manifold with non-negative Ricci curvature and uniformly positive scalar curvature. In fact, \cite{Zhu_Geometryofpsc}*{Corollary 1.9} can only be applied as $n=3$; as $n \geq 4$, we still do not understand the relation between the volume growth of geodesic balls and the splitting of lines at infinity. 

\bigskip
 Finally, we characterize the volume growth of the geodesic balls on three-dimensional manifolds. For related work in three-dimensional dimension, see \cites{Zhu_Geometryofpsc, Chodosh_Chao_volumegrowth, MW_geometryofpsc, Petrunin_ScalarIntegral}. We define 
\[
V_g(r) = \sup_{p \in M} \vol_g(B(p,r)).
\]

\begin{proposition}\label{minimalvolume}
Suppose that $(M^3,g)$ is a complete, non-compact, three-dimensional Riemannian manifold with nonnegative Ricci curvature $\Ric_g \geq 0$ and uniformly positive scalar curvature $\Sc_g \geq 2$. If the sectional curvature $\sec_g \geq -\Lambda^2, \ \Lambda>0$, then, 
\begin{enumerate}
    \item There exists a constant $c_\Lambda\defeq c(\Lambda)$ such that for any $r \geq 0$,
  \[
        V_g(r) \leq c_\Lambda r;
  \]
    \item For any $p \in M$, there exists a universal constant $c_2$ independent of $p$ and $\Lambda$ such that
  \[
        \limsup_{r \rightarrow \infty} \frac{\mr{vol}(B(p,r))}{r} < c_2
    \]
\end{enumerate}
\end{proposition}
Note that compared to previous work \cites{MW_geometryofpsc, Chodosh_Chao_volumegrowth}, our theorem has an extra assumption of sectional curvature lower bound, and the constant $c_\Lambda$ is weaker than that in \cite{MW_geometryofpsc}, but our constant $c_2$ does not depend on the choice of reference points, which is stronger than the main result of \cite{Chodosh_Chao_volumegrowth}.

\vspace{5mm}
Now, we summarize the key idea in this article. The difficulty we need to overcome is the lack of a proper notion of scalar curvature or its lower bound in a Ricci limit space. The crucial observation is that for a torical band $\mb T^{n-1}\times I$, any Riemannian metric of positive scalar curvature on it must satisfy some metric inequalities. The metric inequalities themselves do not require differentiability to describe, see section \ref{sec:torical_cube}. However, the intuitive idea of applying those metric inequalities to a (non-collapsed) Ricci limit space does not work. This is because even if a non-collapsed Ricci limit space happens to be a topological manifold, its limit metric still need not to be induced by a Riemannian metric. We bypass this issue by finding local pull-back maps from the limit space to the approximating manifolds that preserve the topology and almost preserve the metric, essentially by metric Reifenberg theorem. 
To exploit this observation, 
notice that the Ricci limit space $X$ has a large number of splitting factors. Using the $\RCD$ theory we can classify the remaining non-splitting factor, so we have complete understanding of the topology of $X$. It can then be seen that a long torical band $\mb T^{n-1}\times I$ can be embedded into $X$. Using the pull-back maps described above, we can pull back $\mb T^{n-1}\times I$ into a manifold with uniformly positive scalar curvature where the metric inequalities can be applied to find a contradiction. In this way, we succeed to bypass the issue of not having a proper notion of scalar curvature or its lower bound in a Ricci limit space. 

In some cases, we do not know how to find pull-back maps that preserve the topology. For example when singularities of the type $\R ^{n-2}\times C(\mb S^1_r)$ present, where $C(\mb S^1_r)$ is a cone over a circle of radius $r<1$. This can be ruled out by imposing an upper bound on Ricci curvature in all dimension as mentioned in Remark \ref{rem:codim4}. In dimension $3$ to $8$, we can avoid finding topology preserving maps by using the $\mu$-bubbles, this technique has been used extensively in \cites{Gromov_Zhu_area,Chodosh_Li_Soapbubble,Gromov_four_lectures}.

\bigskip
 This article is organized as follows. In section \ref{sec:prelim} we give an introduction to the $\RCD$ theory and the metric inequalities of positive scalar curvature, along the way we prove some auxiliary results. Section \ref{sec:proofs} is devoted to the proofs of all the theorems in the introduction.

\bigskip
\textbf{Acknowledgements} The authors thank Vitali Kapovitch for his suggestions on the proof of Lemma \ref{lem:classification2D} and Lemma \ref{lem:nc_at_infinity}, and Jikang Wang for informing the authors the references \cites{Wei_universalcover,Ye_bettismooth,Ye_bettiRCD}. Also, the authors would like to thank Wenshuai Jiang, Robin Neumayer, Daniele Semola and Guoliang Yu for their insightful discussions and interests on this topic. The first draft of this article was finished when the fourth author was at the Fields institute, and was supported by Fields postdoctoral fellowship.

\section{Preliminaries}\label{sec:prelim}

\subsection{Introduction to metric measure spaces with Ricci curvature lower bounds}\label{sec:RCD}
Given $K\in \R$ and $N\in [1,\infty]$, an $\RCD(K,N)$ space is a complete and separable metric measure space that has synthetic Ricci curvature lower bound $K$ and dimension upper bound $N$. The notion of dimension here can be taken as the Hausdorff dimension \cite{SturmCDII}*{Corollary 2.5}. For example, a complete $n$-dimensional Riemannian manifold with Ricci curvature lower bound $K$ is an $\RCD(K,n)$ space. However, in general, the dimension upper bound $N$ need not be an integer. In this article, we only consider the case $N<\infty$.
The class of $\RCD$ spaces is an intrinsic generalization of the class of Ricci limit spaces, in the sense that the definition of $\RCD$ spaces does not appeal to any approximating sequence and a Ricci limit space obtained as a pointed-measured-Gromov-Hausdorff limit of sequence of $N$-dimensional manifolds with Ricci curvature lower bound $K$ is an $\RCD(K,N)$ space, which follows from the stability of $\RCD(K,N)$ condition under {\rm pmGH} convergence, see for example \cite{GMS13}. 

The structure theory of $\RCD$ spaces stems from the splitting theorem of Gigli \cite{Gigli_splitting}. We will also make an extensive use of this theorem, so we recall it as follows.

\begin{theorem}\label{thm:RCDsplitting}
    Let $N\in [1,\infty)$ and $(X,\dist,\meas)$ be an $\RCD(0,N)$ space. If $X$ contains a line, then there is an measure preserving isometry from $(X,\dist,\meas)$ to $(X',\dist',\meas')\times (\R,|\cdot|, \leb^1)$, such that
    \begin{enumerate}
        \item when $N\in [1,2)$, $X'$ is a point;
        \item when $N\ge 2$, $(X',\dist',\meas')$ is a $\RCD(0,N-1)$ space.
    \end{enumerate}
Here isometry means that if $x=(x',t)$ and $y=(y',s)$, then $\dist(x,y)^2=\dist'(x',y')^2+|t-s|^2$.
\end{theorem}

Let $(X,\dist,\meas)$ be an $\RCD(K,N)$ space, where $\meas$ is a Radon measure with $\mr{supp}\ \meas=X$, we say $(X,\dist,\meas)$ is a non-collapsed $\RCD(K,N)$ space, 
if $\meas=\haus^N$, then necessarily $N\in \N$.

\begin{remark}\label{rmk:notRicciLimit}
In general, for a Ricci limit space $X$ obtained as a pmGH limit of a sequence of $n$-dimensional manifolds with non-negative Ricci curvature, if it splits isometrically as $\R^{k}\times Y$ for some positive integer $k<n$, to our best knowledge it is not known in general if $Y$ is a Ricci limit space approximated by $(n-k)$-dimensional manifolds, but with $\RCD$ theory, especially the splitting theorem above, we can easily see that $Y$ with the metric and measure from the splitting procedure is $\RCD(0, n-k)$, and if $X$ is non-collapsed, so is $Y$. 
\end{remark}

 For a sequence of non-collapsed $\RCD(K,N)$ spaces with uniform positive lower bound on the volume of unit balls i.e.\ the so-called non-collapsing sequence, we have the following volume convergence theorem from \cites{Cheeger-Colding97I,Colding97} and \cite{DPG17}*{Theorem 1.2}. 

\begin{theorem}\label{thm:pGH}
Let $K\in\R$, $N\in \N$ and $(X_i, \dist_i, \haus^N, x_i)$ be a sequence of (pointed) non-collapsed $\RCD(K, N)$ spaces. Then after passing to a subsequence, there exists a (pointed) proper geodesic space $(X, \dist, x)$ such that 
\[
(X_i, \dist_i, x_i) \xrightarrow{\mathrm{pGH}} (X, \dist, x).
\]
Moreover, if $\inf_i\haus^N(B_1(x_i))>0$, then $(X, \dist,\haus^N, x)$ is also a pointed non-collapsed $\RCD(K, N)$ space and the convergence of the $(X_i, \dist_i, \haus^N, x_i)$ to such space is in the ${\rm pmGH}$ topology.
\end{theorem}

As a consequence, non-collapsed Ricci limit space are non-collapsed $\RCD$ spaces. For any integer $N\ge 1$, we say a non-collapsed Ricci limit space is $N$-dimensional if it is a p(m)GH limit of a sequence of $N$-dimensional manifolds. 
Meanwhile, unlike (classical) non-collapsed Ricci limit spaces, non-collapsed $\RCD$ spaces can have boundary. To make this statement rigorous, we first introduce the notion of regular and singular set and then define the boundary of a non-collapsed $\RCD(K,N)$ space. From now on we fix an $\RCD(K,N)$ space $(X,\dist,\meas)$.

\begin{definition}
     Let $\mathcal{R}_k$ be the set of points in $X$ at which the tangent cone is isometric to $(\R^k,\dist_{\R^k},\leb^k)$, for $k\in [1,N]\cap \N$. $\mathcal{R}\defeq \cup_k \mathcal{R}_k$ is called the regular set of $X$. The singular set $\mathcal{S}$ is the complement of the regular set, i.e., $\mathcal{S}\defeq X\setminus \mc{R}$.
\end{definition}

 It is well-known that $\mathcal{S}$ has measure $\meas$-measure zero, see \cite{MN14}*{Theorem 1.1} (after \cite{Cheeger-Colding97I}). When $\meas=\haus^N$, i.e.\ $(X,\dist, \haus^N)$ is a non-collapsed $\RCD(K,N)$ space, we have $\mathcal{R}=\mc{R}_N$ and that the singular set $\mathcal{S}$ is stratified into 
\[
\mathcal{S}^0\subset \mathcal{S}^1\subset \cdots\subset \mathcal{S}^{N-1},
\]
where for $0\le k\le N-1$, $k\in \N$, 
$$
\mathcal{S}^k=\{x\in \mathcal{S}: \text{no tangent cone at $x$ is isometric to } \R^{k+1}\times C(Z)\text{ for any metric space } Z\},
$$ 
where $C(Z)$ is the metric measure cone over a metric space $Z$. It is proved in \cite{DPG17}*{Theorem 1.8} that
\begin{equation}\label{eq:sing}
    \dim_{\haus}(\mathcal{S}^k)\le k.
\end{equation}
We are in position to define the boundary.

\begin{definition}[\cite{DPG17}*{Remark 3.7}]
The De Philippis-Gigli boundary of $(X,\dist,\haus^N)$ is defined as
\begin{equation}\label{eq:DPGboundary}
    \partial X\defeq \overline{\mathcal{S}^{N-1}\setminus \mathcal{S}^{N-2}}.
\end{equation}
\end{definition}

$N$-dimensional Alexandrov spaces with curvature lower bound $K$ are non-collapsed $\RCD(K(N-1), N)$ spaces, and there is also an intrinsic notion of boundary for Alexandrov spaces. It is shown in \cite{BNS20}*{Theorem 7.8} that when an Alexandrov space is viewed as a non-collapsed $\RCD$ space, the two notions of boundary coincide.
On the other hand, it is shown in \cite{Cheeger-Colding97I}*{Theorem 6.2} that if $(X, \dist,\haus^N)$ is a non-collpased Ricci limit space, then $\mc{S}^{N-1}(X)\setminus \mc{S}^{N-2}(X)=\emptyset$, it means a non-collpased Ricci limit space has empty boundary. Under the notion of boundary of a non-collapsed $\RCD$ space, we can also include manifolds with boundary into our study that gives rise to the class of Ricci limit spaces with boundary. This is a subclass of non-collapsed $\RCD$ spaces. 

\begin{definition}\label{def:RicLimBd}
A noncollapsed Ricci limit space with boundary is a ${\rm pmGH}$ limit of (pointed) complete manifolds of the same dimension, with smooth convex boundary and uniform Ricci curvature lower bound in the interior.
\end{definition}

Here, the boundary is said to be convex if the second fundamental form of the boundary is non-negative definite. We refer the readers to \cite{BNS20} for a comprehensive treatment of the boundary of non-collapsed $\RCD$ spaces. The second fundamental form of the boundary is defined as $ - \nabla \nu $ where $\nu$ the inner unit normal vector field. With this convention, $n$-unit sphere $\mathbb{S}^n$ in $\mathbb{R}^{n+1}$ has second fundamental form $A=I_n$. 

\begin{remark}
From the computation carried out by Han, we can infer that if a manifold with boundary is an $\RCD(K,N)$ space (hence also a $\RCD(K,\infty)$ space), then the boundary is necessarily convex, see \cite{Han20}*{Corollary 2.5}.
\end{remark}

It is worth pointing out that the class of non-collapsed $\RCD(K,N)$ spaces without boundary is strictly larger than the class of non-collapsed Ricci limit spaces. For example $C(\R P^2)$ is a non-collapsed $\RCD(0,3)$ space due to Ketterer \cite{Ketterer_cone}, but this is not a $3$-dimensional topological manifold, since $C(\R P^2)\setminus \{\text{cone tip}\}$ deformation retracts to $\R P^2$ while $\R^3\setminus\{0\}$ deformation retracts to $\mb S^2$, whereas Simon-Topping in \cite{SimonTopping_Ricci} proved that any $3$-dimensional non-collapsed Ricci limit space is a topological manifold. 

When a non-collapsed Ricci limit space has no singular points, we have the following topological stability theorem from \cite{KapMon19}*{Theorem 3.3} (after \cite{Cheeger-Colding97I}*{Appendix A} which addresses the compact case). 

\begin{theorem}\label{lem:top}
Let $(M_i,g_i, q_i)$ be a sequence of complete Riemannian manifolds with $\mr{Ric}_{g_i}\ge 0$, and $\inf_i \mr{vol}_{g_i}(B(q_i,1))>\nu>0$ such that
$$
(M,g, q_i)\xrightarrow{\mr{pGH}} (X,\dist, q_\infty),
$$ 
where $(X,\dist)$ satisfies $\mc{R}(X)=X$, then for any fixed $\alpha\in (0,1)$, any $R>0$, any $p_i\to p_\infty$ along with the pGH convergence and a large $i$, there exists $\alpha$-bi-H\"older embedding $f_i: B(p_i, R) \rightarrow B(p_\infty, R+\epsilon_i)$ which is also an $\epsilon_i$-GH approximation with $\epsilon_i \rightarrow 0$ and $\dist_{g_i}(f_i(p_i), p_\infty) \leq \epsilon_i$.
\end{theorem}

Now we introduce two classification results of $n$-dimensional non-collapsed Ricci limit spaces that split off $\R^{n-1}$ or $\R^{n-2}$. The proofs depend on the classification of non-collapsed $\RCD(0,1)$ spaces, the classification of $2$-dimensional Alexandrov spaces and the fact that non-collapsed Ricci limit spaces do not have boundary. However, there is some subtlety, see Remark \ref{rmk:notRicciLimit}. The first one is essentially the classification of non-collapsed $\RCD(0,1)$ spaces.

\begin{lemma}\label{lem:classification1D}
Let $n\ge 2$ be an integer, and $(X,\dist,\haus^n)$ be a non-collapsed Ricci limit space which is also a non-collapsed $\RCD(0,n)$ space. If $X$ splits isometrically into $\R^{n-1}\times Y$, then $Y$ is isometric to either $S^1$ or $\R$, so $X$ is isometric to $\R^n$ or $\R^{n-1}\times \mb{S}^1$.
\end{lemma}

\begin{proof}
 By iteratively applying Theorem \ref{thm:RCDsplitting}, along with the fact that the reference measure in $X$ is $\haus^n$, $(Y,\dist_Y,\haus^1)$ is a non-collapsed $\RCD(0,1)$ space, where $\dist_Y$ satisfies $\dist^2=\dist_Y^2+\dist_{\R^{n-1}}^2$. The classification of $1$-dimensional $\RCD$ spaces \cite{KL15}*{Theorem 1.1} implies that $(Y,\dist_Y)$ is isometric to either $\R$, $\mb{S}^1$ or an interval $[0,b]$, $b\in (0,\infty]$. However, $Y$ cannot be $[0,b]$. If this was the case then $X$ would have nonempty $\mc{S}^{n-1}(X)\setminus \mc{S}^{n-2}(X)$, 
a contradiction to \cite{Cheeger-Colding97I}*{Theorem 6.2} we recalled earlier.
\end{proof}

To proceed we need the following.

\begin{theorem}[\cite{LytchakStadler_Ricci_to_Alex}]\label{thm:RCD02=Alex}
   Every non-collapsed $\RCD(0,2)$ space is an $2$-dimensional Alexandrov space of non-negative curvature.
\end{theorem}

 The next classification theorem essentially concerns $2$-dimensional Alexandrov spaces without boundary. The result is well-known to Alexandrov geometry community but seems not stated in the literature. 

\begin{lemma}\label{lem:classification2D}
Let $(X,\dist,\haus^n)$ be an $n$-dimensional Ricci limit space, if $X$ splits isometrically as $X=\R^{n-2}\times Y$, with $Y$ being compact, then $(Y,\dist_Y)$ is isometric to one of the following spaces equipped with a metric of non-negative curvature in the sense of Alexandrov space with homeomorphism type as $(1)$ a Klein bottle $\mb{K}$; $(2)$ a torus $\mb{T}^2$; $(3)$ a real projective space $\R P^2$; $(4)$ a sphere $\mb{S}^2$.
\end{lemma}

\begin{proof}
First observe that same as the proof of Lemma \ref{lem:classification1D}, it follows from Theorem \ref{thm:RCDsplitting} that $Y$ is a non-collapsed $\RCD(0,2)$ space. Moreover, since $X$ has no boundary, $Y$ also has no boundary. 

The statement about metric is clear due to Theorem \ref{thm:RCD02=Alex}. It suffices to consider the topological type. It is well-known that a compact $2$-dimensional Alexandrov space without boundary is topologically a closed $2$-dimensional surface. Consider the universal cover $\tilde Y$ of $Y$, if $\tilde Y$ splits off $\R^2$, then $\tilde Y$ is isometric to $\R^2$, so $Y$ is locally isometric to $\R^2$, thus a flat, closed $2$-dimensional manifold, the only possible ones are $\mb{T}^2$ or $\mb K$. If $\tilde Y$ splits off $\R$ but not $\R^2$, this is impossible by the classification Lemma \ref{lem:classification1D}. If $\tilde Y$ does not split, then $\tilde Y$ is a simply connected closed surface so it must be $\mb S^2$,  then $Y$ is $\mb S^2$ or $\R P^2$.

\end{proof}

The following lemma says that one can always find a non-collapsed limit at infinity in a $2$-dimensional Alexandrov space without boundary. We mainly use the gradient flow of a Busemann function in an Alexandrov space, so we refer to Petrunin \cite{Petrunin_semi} for the theory of gradient flow of semiconcave functions in Alexandrov spaces.

\begin{lemma}\label{lem:nc_at_infinity}
Let $(X,\dist,\haus^2)$ be a non-compact non-collapsed $\RCD(0,2)$ space without boundary, then there exists a sequence of points $p_i\to \infty$ so that $(X,\dist,p_i)$ pGH converges to either $\R^2$ or $\R^1\times \mb{S}^1$.
\end{lemma}

\begin{proof}
Since $X$ is non-compact, there exists a unit speed ray $\gamma(t):[0,\infty)\to X$ emanating from some point $p\in X$, i.e. $\gamma(0)=p$. Let 
\begin{equation}
    b_{\gamma}(x)\defeq \lim_{t\to \infty}t-\dist(x,\gamma(t)),
\end{equation}
be the associated Busemann function. Note that $X$ is an Alexandrov space with non-negative curvature due to Theorem \ref{thm:RCD02=Alex}. By triangle comparison, $b_{\gamma}(x)$ is convex. 

Take $p_i\defeq\gamma(t_i)$ for some sequence of positive numbers $t_i\nearrow \infty$. We claim that any limit space $(Y,y)$ of $(X,p_i)$ must be either $\R^2$ or $\R\times\mb S^1$. First notice that if we reparametrize $\gamma$ as $\tilde\gamma:[-t_i,\infty)$, so that $\tilde \gamma(0)=p_i=\gamma(t_i)$, then we see that $\tilde \gamma$, hence $\gamma$, locally uniformly converges to a line in $Y$ along $(X,p_i)\to (Y,y)$. By the splitting theorem for $\RCD$ spaces, Theorem \ref{thm:RCDsplitting}, $Y$ splits off an $\R$-factor. 

We argue by contradiction. Write $Y=\R\times K$, where $K$ is a $\RCD(0,1)$ space. If $Y$ is not $\R^2$ or $\R\times\mb S^1$, by the classification theorem of $\RCD(0,1)$ spaces, $K$ can only be a single point or an interval. By the stability of vanishing of boundary \cite{BNS20}*{Theorem 1.6}, $K$ cannot be an interval since $X$ does not have boundary. So $K$ must be a single point and $Y$ is isometric to $\R$ with Euclidean metric. Assume that $(X,p_i)$ pGH converges to $(\R,0)$. Consider the function $b_\gamma-t_i$. It is a locally bounded $1$-Lipschitz function, so by Arzela-Ascoli theorem passing to a subsequence it locally uniformly converges to a $1$-lipschitz function $b$ on $(Y,y)=(\R,0)$, see for example \cite{MN14}*{Proposition 2.12}. Notice that for any $a\in \R$, $b_\gamma(\gamma(t_i+a))-t_i=a$. Combining this with the local uniform convergence $\gamma\to \R$ along $(X,p_i)\to (\R,0)$, we infer that $b$ is the identity function on $\R$. 

For the sake of contradiction, we aim to prove that 
\begin{enumerate}[label=(\textbf{L.\arabic*})]
    \item \label{eq:connected} For any $s\in \R$, $\{b_\gamma=s\}$ is path connected by Lipschitz curves.
    \item \label{eq:nondecreasing} The function $s\mapsto \haus^1(\{b_{\gamma}= s\})$ is non-decreasing as $s$ increases.
\end{enumerate}
If \ref{eq:connected} and \ref{eq:nondecreasing} are proved, then we see from \ref{eq:nondecreasing} that there exists $c_0>0$,  so that $\mr{diam}(\{b_\gamma=t_i\})\ge 2c_0>0$. If this is not true then $\liminf_{i\to\infty}\mr{diam}(\{b_\gamma=t_i\})=0$, which in turn implies $\liminf_{i\to\infty}\haus^1_\infty(\{b_\gamma=t_i\})\to 0$. It follows $\liminf_{i\to\infty}\haus^1(\{b_\gamma=t_i\})\to 0$. This is a contradiction to $\haus^1(\{b_\gamma=t_i\})\ge \haus^1(\{b_\gamma=0\})>0$. 
We then see from \ref{eq:connected} and the intermediate value theorem that there exists $q_i\in \{b_\gamma=t_i\}$ so that $\dist_g(p_i,q_i)= c_0$. We already know $p_i\to 0$ along $(X,p_i)\to (Y,0)$. Assume that $q_i\to a\in \R$, this is possible because $q_i$ has uniformly bounded distance to $p_i$. However $0=b_\gamma(p_i)-t_i=b_\gamma(q_i)-t_i\to b(a)=a$, which means $p_i,q_i$ converges to the same point in $\R$, this is a contradiction to $\dist_g(p_i,q_i)=c_0>0$. 

We now prove \ref{eq:connected} and \ref{eq:nondecreasing} by the gradient flow of some truncation of $b_\gamma$.

For \ref{eq:connected}. Let $s\in \R$ and $\sigma$ be a geodesic joining two points on $\{b_\gamma=s\}$. By the convexity of $b_\gamma$, $\sigma\subset \{b_\gamma\le s\}$. Set $\bar b_\gamma=\min\{ -s, -b_\gamma\}$, it is a minimum of concave functions hence also concave, and $\sigma\subset\{\bar b_\gamma\ge -s\}$. Denote the gradient flow of $\bar b_\gamma$ by $G_t:X\to X$ for $t\in[0,\infty)$, whose existence is guaranteed by \cite{Petrunin_semi}*{Proposition 2.1.2}. Since $\sigma$ is compact, and by construction $G_t$ fixes $\{-b_\gamma\le -s\}$ for every $t\in[0,\infty)$, there exists $T>0$ so that $G_T(\sigma)\subset \{b_\gamma=-s\}$. so $G_T(\sigma)$ is the desired Lipschitz curve in the level set that joins the given two pints.

For \ref{eq:nondecreasing}. given two level sets $\{b_{\gamma}= s_1\}$ and $\{b_{\gamma}= s_2\}$ with $s_2>s_1$. To show $ \haus^1(\{b_{\gamma}= s_2\})\ge  \haus^1(\{b_{\gamma}= s_1\})$, it suffices to show that there exists a $1$-Lipschitz surjective map from $\{b_{\gamma}= s_2\}$ to $\{b_{\gamma}= s_1\}$, see also \cite{antonelli2023isoperimetric}*{Remark 3.3, item 3}. 

As before we set $\hat b_\gamma=\max\{ s_1, b_\gamma\}$ which is still convex. Consider the gradient flow $F_t:X\to X$ of $\nabla \hat b_{\gamma}$, defined for all $t\in [0,\infty)$. The convexity of $\hat b_\gamma$ implies that 
\begin{equation}
    \dist(F_t(x),F_t(y))\le \dist(x,y), \ t\in[0,\infty),
\end{equation}
which amounts to saying that $F_t$ is $1$-Lipschitz for any $t\in[0,\infty)$. By construction $F_t$ fixes $\{b_\gamma\le s_1\}$. Meanwhile since the Lipschitz constant of $b_\gamma$, $\mr{lip}~ b_\gamma=1$ everywhere on $X$, we see that $F_{s_2-s_1}:\{b_\gamma=s_2\}\to \{b_\gamma=s_1\}$ is the desired map. Indeed, we only need to verify that it is surjective. This follows from \cite{Petrunin_semi}*{Section 2.2, Property 3} (see also \cite{Fujioka_gradientflow}*{Lemma 6.1} for a statement with no assumption on time $t$ being small).
\end{proof}

\begin{remark}
In fact, in the setting of Lemma \ref{lem:nc_at_infinity}, a stronger statement is recently proved by Antonelli-Pozzetta in \cite{antonelli2023isoperimetric}*{Theorem 4.2} that for any $p_i\to \infty$, $(X,\dist, p_i)$ is a non-collapsing sequence. We can also give an alternative argument by running the same proof as that of lemma \ref{lem:nc_at_infinity}, provided that a non-compact Alexandrov space always splits off a line at the infinity, see \cite{antonelli2023isoperimetric}*{Lemma 2.29} and \cite{zhu2023twodimension}*{Corollary 4.3}. We give a sketch of proof as it is not needed.

Indeed, fix a base point $p_0$, choose a geodesic $\gamma_i$ joining $p_0$ and $p_i$, by passing to subsequence $\gamma_i$ locally uniformly converges to a ray $\gamma$. It is shown in \cite{antonelli2023isoperimetric}*{Lemma 2.29} that $\dist_g(p_i,\gamma)=o(\dist_g(p_0,p_i))$. Let $\bar p_i$ be a choice of closest point on $\gamma$ to $p_i$, then we see that $b_\gamma (p_i)\ge b_\gamma (\bar p_i)-\dist(\bar p_i,p_i)=\dist(\bar p_i,p_0)-\dist(p_i,\gamma)\ge \dist( p_i,p_0)-2\dist(p_i,\gamma) \to \infty$. If $(X,p_i)$ is collapsing, as shown before it must collapses to $(\R,0)$. Then we know $b_\gamma-b_\gamma(p_i)$ locally uniformly converges to the identity of $\R$. from \ref{eq:connected} and \ref{eq:nondecreasing} we infer again that there is $c_0>0$ and $q_i$ such that $\dist_g(p_i,q_i)=c_0$ and $b_\gamma(p_i)=b_\gamma(q_i)$, and this is enough for the contradiction. 

\end{remark}

\subsection{Torical band and cube inequality}\label{sec:torical_cube}
We recall some metric inequalities for positive scalar curvature on $2$ model spaces, the torical band and the cube. The authors believe that the cube inequality or torical band could find more applications in singular spaces.

Let $M^{n}$ be homeomorphic to $\mb T^{n-1} \times I$, where $I$ is an interval $[a,b]$, $a < b$. Such an $M$ is called a torical band. We define
$$
\partial M= \mb T^{n-1}\times \{b\} \bigcup \mb T^{n-1} \times \{a\} = : \partial_+ M \bigcup \partial_-M,
$$
and for any Riemannian metric $g$ on $M$, set
\begin{equation}
    \dist_g(\partial_+ M, \partial_-M) \defeq\inf \{ \dist_g(x,y):{x \in \partial_+ M, y \in \partial_- M }\}.
\end{equation}
Then, the following theorem holds.

\begin{theorem}[G-WXY \cites{Gromov_metric_inequality, Gromov_four_lectures,xiepsc, WXY_cubeinequality}\label{toricalband}]
If $(M,g)$ is a $n$-dimensional torical band with $Sc_g \geq n(n-1)$, then
\begin{equation}
     \dist_g(\partial_+ M, \partial_-M) <\frac{2\pi}{n}.
\end{equation}
\end{theorem}

\begin{lemma}\label{lem:toricband}
    Let $n\in \N^+$, a toric band $\mb T^{n-1}\times I$ can be topologically embedded into $\R^n$. Equip $\R^n$ with standard Euclidean metric and restrict it to (the interior of) $\mb T^{n-1}\times I$, then by rescaling the metric, the length of the interval $I$ can be arbitrarily large.
\end{lemma}
\begin{proof}
   Take a Clliford torus $\mb T^{n-1}\subset \R^n$ , i.e., a torus of revolution parametrized by a sequence of increasing radii $0<r_1<r_2<\ldots< r_{n-1}$. The $r_0$-tubular neighborhood of this $\mb T^{n-1}$ for $r_0= \frac{r_1}4$ is a possible topological embedding.
\end{proof}

\bigskip

Let $D^n$ be homeomorphic to $[-1,1]^n$ and $\partial_{i+},\partial_{i-}$ be a pair of opposite faces in $D$, $i=1,2,\ldots, n$. Then, we have 
\begin{theorem}[\cites{Gromov_metric_inequality, WXY_cubeinequality}]\label{thm:cubeineq}
If the cube $(D,g)$ satisfies $\Sc_g\ge 4\pi^2 n(n-1)$, then
\begin{equation}\label{eq:cubeineq}
    \min_i \dist_g (\partial_{i+},\partial_{i-}) < \frac{1}{\sqrt n}.
\end{equation}
\end{theorem}

This is a special case of \cite{WXY_cubeinequality}*{Theorem 1.1}. In general, we can also consider manifolds that admit a nice map into a possibly lower dimensional cube. This leads us to Gromov's $\square^{n-m}$ inequality. The next theorem is a hugely simplified version of Gromov's $\square^{n-m}$ inequality proved in \cite{WXY_cubeinequality}*{Theorem 1.2}, which is sufficient for our use.

\begin{theorem}\label{thm:cube(n-m)}
    If $(X=\mb S^1\times \Pi_{i=1}^n [-R_i,R_i],\ g)$ is a $(n+1)$-manifold with corners and  $Sc_g(X)\ge k>0$, then
    \[
    \sum_{i=1}^n\frac{1}{ R_i^2}\ge \frac{kn}{\pi^2 (n-1)}
    \]
\end{theorem}

\begin{proof}
    We apply \cite{WXY_cubeinequality}*{Theorem 1.2} to the map $P: \mb S^1\times \Pi_{i=1}^n [-R_i,R_i]\to [-1,1]^n$ that kills the $\mb S^1$ factor by projection and rescales the intervals. The conditions listed in \cite{WXY_cubeinequality}*{Theorem 1.2} are trivially satisfied for $P$.
\end{proof}

Theorem \ref{toricalband} and \ref{thm:cubeineq} indicate that uniformly positive scalar curvature can control the size of manifolds and play a vital role in the proofs of Theorem \ref{thm:limit_n-2} and \ref{thm:gap}.

\section{Proofs of Main Theorems}\label{sec:proofs}
Let us start by proving theorem \ref{thm:limit_n-2}.
\begin{proof}




First, let us prove \eqref{main:item1} as follows. If $X$ splits $\R^{n-1}$, thanks to Lemma \ref{lem:classification1D}, we have $X=\mb{R}^n$ or $X=\mb{S}^1 \times \mb{R}^{n-1}$. In fact, $\mb S^1$ with non-collapsed $\RCD(0,1)$ structure is flat, so both $\mb{S}^1 \times \mb{R}^{n-1}$ and $\mb{R}^n$ are smooth Riemannian manifolds. By lemma \ref{lem:toricband}, we can find a large torical band $\mr{Bd} \defeq \mb T^{n-1} \times [0, 11] \subset \mb S^1 \times \mathbb{R}^{n-1}$ or $\mr{Bd} \subset \R^n$. 
There exists a radius $r>1$ satisfying 
\[ 
\mr{Bd} \subset B^X(p, r-1) \subset X. 
\] 
By Theorem \ref{lem:top}, there exist $\alpha$-bi-H\"older homeomorphisms $f_i: B(p_i, r) \rightarrow B(p, r+\epsilon_i)$ that are also $\epsilon_i$-GH approximations with $\epsilon_i \rightarrow 0$ and $d_g(f_i(p_i), p) \leq \epsilon_i$. We pull back the torical band $\mr{Bd}$ from the Ricci limit space to $(M_i,g_i,p_i)$, then $\mr{Bd}_i = f_i^{-1}(\mr{Bd}) \subset B(p_i, r)$ for sufficiently large $i$. After an arbitrarily small perturbation, we may assume that $\mr{Bd}_i$ is a smooth domain in $M_i$. Hence, we have constructed torical bands with restriction Riemannian metric $(\mr{Bd}_i,g_i)$ satisfying
$$
\dist_{g_i}(\partial_+\mr{Bd}_i, \partial _- \mr{Bd}_i)\ge 10, \ \Sc_{g_i} \geq 2.
$$
This contradicts the torical band estimate by Theorem \ref{toricalband} (or see \cites{Gromov_metric_inequality, Gromov_four_lectures,xiepsc, WXY_cubeinequality}), so we rule out the two cases $X=\mb{R}^n$ or $X=\mb{S}^1 \times \mb{R}^{n-1}$. This completes the proof of \eqref{main:item1}.

Next, let us prove \eqref{main:item2} as follows. Note that Lemma \ref{lem:nc_at_infinity} implies that  $Y$ is compact. Indeed, if $Y$ is not compact, by Lemma \ref{lem:nc_at_infinity} we can find points $p_i\in Y$ so that $(Y,\dist_Y, p_i)$ pGH converges to $\R^2$ or $\R\times\mb{S}^1$, which means by a diagonal argument we can re-choose a sequence $q_i\in M_i$ so that $(M_i, g_i, q_i)$ converges to $\R^n$ or $\R^{n-1}\times\mb{S}^1$, which  contradicts \eqref{main:item1} that has been proved above.

Now that $Y$ is compact, the classification in Lemma \ref{lem:classification2D} implies that $Y$ is one of $\mb{K}$, $\mb{T}^2$, $\R P^2$ or $\mb{S}^2$. To rule out $\mb{T}^2$, first notice that when equipped with a non-collapsed $\RCD(0,2)$ structure, the corresponding metric on $\mb{T}^2$ is necessarily flat hence smooth, see \cite{MondinoWeiUni}*{Corollary 1.4}. 
Then also notice that a large torical band $\mb{T}^{n-1}\times [0,11]$ can be embedded into $\mb{T}^2\times \R^{n-2}$, the proof goes exactly the same as in the proof of \eqref{main:item1}. For $\mb K$, note that $\mb{T}^2$ is the orientable double cover of $\mb{K}$, for large $i$ we can find $f_i:[0,11]\times \mb T^{n-3}\times\mb K\to M_i$ which is a homeomorphism onto its image and an $\epsilon_i$-GH approximation. Restrict the  Riemannian metric on $M_i$ to the image of $f_i$ and lift the image of $f_i$ to its orientable double cover, we find a metric of positive curvature lower bound $2$ on $[1,10]\times T^{n-1}$, which is impossible. 

Finally, let us prove (\ref{main:item3}). Suppose that $D=\text{diam}(\mathbb{S}^2, \dist_{\mathbb{S}^2})$ and $n,s \in \mathbb{S}^2$ such that $\dist(n,s)=D$. If we remove a small spherical cap 
 $B(n,\epsilon$) (resp. $B(s,\epsilon)$) centered at $n$ (resp.  $s$) of radius $\epsilon >0$, then we obtain that $\mathbb{S}^2- (B(n,\epsilon) \cup B(s,\epsilon))$ has topological type $\mathbb{S}^1 \times I $ and the distance between the boundary components $\partial B(n, \epsilon)$ and $\partial B(s, \epsilon)$ is $D-2\epsilon$. For any $R>0$, there exists an $\epsilon$-GH approximation $f_i: B^{M_i}(p_i,R)\to B^X(p,R+\epsilon)$.
\begin{itemize}
  \item 
  If (a) is satisfied, then it is valid to apply Theorem \ref{lem:top}. We get that $f_i$ is a homeomorphism, so when $R>D$, $f_i^{-1}([-R/2,R/2]\times \mathbb{S}^1 \times I )$ is still topologically a product of $\mb S^1$ and intervals, which we denote by $ \mb S^1_M\times I_M\times [-R/2,R/2]^{n-2}_M$. Without loss of generality, we may assume that $\mb S^1_M\times I_M\times [-R/2,R/2]^{n-2}_M$ is a smooth domain in $M$, then  
we deduce by Theorem \ref{thm:cube(n-m)},
\begin{equation}
    \frac{1}{(D-3\epsilon)^2}+\frac{n-2}{(2R-\epsilon)^2}>\frac n{2\pi^2(n-1)}. 
\end{equation}
Let $R\to \infty$ then let $\epsilon\to 0$, we derive that 
 \begin{equation}
     D\le \sqrt{\frac{2(n-1)}n}\pi
 \end{equation}
as desired.

 \item If (b) is satisfied, i.e.,  $3 \leq n \leq 8$.
 We let $N\defeq f^{-1}_i(\mathbb{S}^2- (B(n,\epsilon) \cup B(s, \epsilon)) \times [-R/2,R/2]^{n-2})$, where $\mathbb{S}^2- (B(n,\epsilon) \cup B(s, \epsilon)) \times [-R/2,R/2]^{n-2}$ is homeomorphic to $S^1\times I \times [-R/2,R/2]^{n-2})$. Without loss of generality, we may assume that $N$ is a smooth domain in $M$, and the distance among the opposites sides are $D-2\epsilon, R, \cdots, R$ since $f_i$ is an $\epsilon$-GH approximation. Here, the opposite sides refers to the pairs 
 \[
 \left\{(  f_i^{-1}(\partial B(n,\epsilon)), f_i^{-1}(\partial B(s,\epsilon) );\ \  ( f_i^{-1}(\partial_{j+}), f^{-1}_i(\partial_{j-}) ) \right\},
 \] 
where $\partial_{j+}, \partial_{j-}$ are faces of the cube $I\times [-R/2,R/2]^{n-2}$, $j=1,2,\ldots n-1$. Next, we will run the standard $\mu$-bubble techniques due to \cite{Gromov_four_lectures}*{Page 260} and \cite{Gromov_Zhu_area}*{Theorem 1.1}.  

 Without loss of generality, we may assume that $ f_i^{-1}(\partial B(n,\epsilon))\cup f_i^{-1}(\partial B(s,\epsilon)$ intersects $\partial N - (f_i^{-1}(\partial B(n,\epsilon)) \cup f_i^{-1}(\partial B(s,\epsilon)))$ in acute angle by a small perturbation. Now we construct  a smooth map
$$\varphi_1: N \rightarrow \left[-\frac{D-\epsilon}{2}, \frac{D-\epsilon}{2}\right]$$
such that $|d\varphi_1| \leq 1$ and $\varphi^{-1}_1( \frac{D-\epsilon}{2}) =f_i^{-1}(\partial B(n,\epsilon)) $ and $\varphi^{-1}_1(- \frac{D-\epsilon}{2}) =f_i^{-1}(\partial B(s,\epsilon))$. Now 
we define $$N_0 = \{x \in M: \varphi_1(x) <0 \}$$ and then consider the class
$$\mathcal{C} = \left\{\text{Caccioppoli sets } \Omega \text{ in } X \text{ such that } \Omega \Delta \Omega_0 \Subset N \setminus (f_i^{-1}(\partial B(n,\epsilon))\cup f_i^{-1}(\partial B(s,\epsilon)))\right\}.$$
Moreover, we set 
$$h_1: \left(-\frac{D-\epsilon}{2}, \frac{D-\epsilon}{2}\right) \rightarrow \mathbb{R},\ \  t \rightarrow - \frac{2(n-1)\pi}{n(D-\epsilon)}\tan\left(\frac{\pi t}{D-\epsilon}\right)$$
and define
$$\mu(\Omega) = \mathcal{H}^{n-1}(\partial^*\Omega \cap \text{int}(N)) - \int_N (\chi_\Omega - \chi_{N_0})h_1 \circ \varphi_1  d\mathcal{H}^n,$$
where $\partial^*\Omega$ is the reduced boundary of $\Omega$, $\text{int}(N)$ is the interior part of $N$ and $\chi_\Omega$ is the characteristic function of $\Omega$. Since $ 3 \leq n \leq 8$, there exists a smooth minimizer $\Omega_0$ of $\mu$ such that(see \cites{Gromov_Zhu_area}) 

\begin{enumerate}

    \item $Y_0 = \partial \Omega_0$ is a smooth embedded hypersurface with free boundary that separates $f_i^{-1}(\partial B(n,\epsilon)$ and $f_i^{-1}(\partial B(s,\epsilon)$.
    
    \item at least one connected component of $Y_0$ intersects with $f_i^{-1}(\partial_{1+}) \cup f_i^{-1}(\partial_{1-})$ and it satisfies $(1)$ above. Hence, we collect such components of $Y_0$ and still denote it by $Y_0$. Moreover, we denote $\partial_{1\pm}Y_0$ by the intersection of $\partial Y_0$ and $f^{-1}(\partial_{1\pm})$. It is clear that $\dist(\partial_{1+}Y_0,\partial_{1-}Y_0) \geq R$.

    \item the stability condition holds: 
    $$
    \qquad\qquad\int_{Y_0} (|\nabla_{Y_0}f|^2 - \frac{1}{2}(\Sc(N)- \Sc(Y_0) + |\AA|^2 + (\frac{n}{n-1}h_1^2 + 2h_1^\prime)\circ \varphi_1)f^2  \geq  \int_{\partial Y_0} A_{\partial N}(\nu, \nu)f^2.
    $$
    Here, $\nu$ is the unit outer normal vector field of $Y_0$ with respect to $\Omega_0$, $\AA$ denotes the trace-free part of the second fundamental form of $Y_0$ in $N$ and $A_{\partial N}$ is the second fundamental form of $\partial N$ with respect to the unit outer normal $\vec{n}$.
\end{enumerate}
Note that $$\frac{n}{n-1}h_1^2 + 2h^\prime_1 = \frac{4(n-1)\pi^2}{n(D-\epsilon)^2}.$$ The stability implies that there exists a positive function $u_0$ on $Y_0$ such that
$$-\Delta_{Y_0} u_0 - \frac{1}{2}(Sc(N) - Sc(Y_0) + |\AA|^2 - \frac{4(n-1)\pi^2}{n(D-\epsilon)^2})u_0 = \lambda_1 u_0 \geq 0, \text{ on } Y_0$$
and $$\frac{\partial u_0}{\partial \vec{n}} = A_{\partial N}(\nu,\nu)u_0,$$
where $\lambda_1$ is the corresponding first eigenvalue.

Next, we consider 
$$(N_1 = Y_0 \times \mathbb{S}^1,\  g_1 = g_{Y_0} + u_0^2 ds^2).$$
A direct calculation implies that for $y_0 \in Y_0,\  \theta \in \mathbb{S}^1$,
$$\Sc(N_1, (y_0, \theta)) = \Sc(Y_0, y_0) - 2u^{-1}_0\Delta_{Y_0}u_0 \geq \Sc(N, y_0) - \frac{4(n-1)\pi^2}{n(D-\epsilon)^2}.$$
Moreover, we construct a smooth function on $N_1$ as follows:
$$\varphi_2: Y_0 \rightarrow \left[-\frac{R}{2}, \frac{R}{2}\right]$$
such that $|d\varphi_2| \leq 1$ and $\varphi^{-1}_2( -\frac{R}{2}) =f_i^{-1}(\partial_{1-})$ and $\varphi^{-1}_2( \frac{R}{2}) =f_i^{-1}(\partial_{1_+} )$. Here, we may view $\varphi_1$ as an $\mathbb{S}^1$ invariant function on $N_1$. A similar argument implies there exists a smooth minimizer $\Omega_1$ for $\mu$-bubble functional and whose boundary $Y_1$ is a smooth embedded hypersurface with free boundary separating $f^{-1}{(\partial_{1+})}$ and $f^{-1}{(\partial_{1-})}$, 
then we consider
$$(N_2 = Y_1 \times T^2, \ g_2 = g_1 + u_1^2ds^2).$$
Here, $u_1$ the positive first eigenfunction of the stability condition of $Y_1$. 
The same direct calculation as above implies that
$$\Sc(N_2, (y_1, \theta))  \geq \Sc(N, y_1) - \frac{4(n-1)\pi^2}{n}\left(\frac{1}{(D-\epsilon)^2} + \frac{1}{R^2}\right)$$
for any $y_1 \in Y_1$ and $\theta \in T^2$. Inductively, we run this argument $n-1$ times to obtain that a smooth closed manifold
$$(N_{n-1} = \mathbb{S}^1 \times \mathbb{T}^{n-1}, \ g_{n-1} = g_{n-2} + u_{n-2}^2 ds^2 )$$
with
\[
\Sc_{g_{n-1}} \geq 2 - \frac{\pi^2(n-1)(n-2)}{n R^2} - \frac{4\pi^2(n-1)}{n(D-2\epsilon)^2},
\] 
Note that $N_{n-1}$ admits no metric with positive scalar curvature, we conclude that
 \[2 - \frac{\pi^2(n-1)(n-2)}{n R^2} - \frac{4\pi^2(n-1)}{n(D-2\epsilon)^2} \leq 0.\]
 Let $R \rightarrow \infty$ and $\epsilon \rightarrow 0$. we derive
 \[D \leq \sqrt{\frac{2(n-1)}{n}}\pi.\]


\end{itemize}

\end{proof}
\begin{remark}
    We provide an alternative proof of item \eqref{main:item1} and \eqref{main:item2} of Theorem \ref{thm:limit_n-2} using the almost splitting map (see \cite{BNS20}*{Definition 3.1} for the definition) that is suggested by the referee. Again, if $X$ is isometric to either of $\R^n$, $\R^{n-1}\times \mb S^1$, $\R^{n-2}\times\mb T^2$ or $\R^{n-2}\times K$, then $\mc R(X)=X$, since $\mb T^2$ or $K$ with $\RCD(0,2)$ structure must be flat. Without loss of generality, we can assume $B_2(p_i)\subset M_i$ pGH converges to $B_2(0)\subset \R^n$ thank to the flatness of $\mb S^1$, $\mb T^2$ and $K$. Fix $\delta>0$ to be determined, by \cite{BNS20}*{Theorem 3.5}, there exists a large $i$ depending on $\delta$, so that there exists a harmonic map $(u_1,\ldots, u_n): B_2(p)\to \R^n$. Here we denote $p\defeq p_i$ for simplicity. Moreover there exists a constant $c(n)>0$ depending only on $n$, so that for any $a,b=1,\ldots, n$ we have
\begin{enumerate}[label=(\textbf{S.\arabic*})]
    \item\label{item:GradientEsti} $\sup_{x\in B_1(p)}|\nabla u_a|\le 1+c(n)\sqrt{\delta}$. See \cite{BNS20}*{Remark 3.3}.
    \item $4\dashint_{B_2(p)} |\hess_{u_a}|^2 d\vol_g<\delta$.
    \item $\dashint_{B_2(p)}|\nabla u_a\cdot \nabla u_b-\delta_{a,b}|d\vol_g<\delta$.
\end{enumerate}
We allow $c(n)$ to change from line to line but it still only depends on $n$. Let $\phi\in C_c^\infty(M_i)$ be a cut-off function so that $\phi\in[0,1]$ on $M_i$, $\phi=1$ on $B_1(p)$ and $\phi=0$ on $M_1\setminus B_2(p)$. This $\phi$ satisfies $|\nabla \phi|+|\Delta \phi|<c(n)$ for some $c(n)>0$. Then by Bochner formula and the harmonicity of $u_a$, $a=1,\ldots,n$, we can estimate
\begin{align}\label{eq:RicEsi}
    \int_{B_1(p)}\Ric(\nabla u_a,\nabla u_b)d\vol_g&\le \int_{B_2(p)}\frac12|\nabla u_a\cdot\nabla u_b-\delta_{a,b}||\Delta \phi|+\phi|\hess_{u_a}||\hess_{u_b}|d \vol_g\\
    &\le \frac{c(n)}2\int_{B_2(p)} |\nabla u_a\cdot\nabla u_b-\delta_{a,b}|+|\hess_{u_a}|^2+|\hess_{u_b}|^2d\vol_g\notag\\
    &\le c(n)\vol_g(B_2(p)) \delta\le c(n) \vol_g(B_1(p)) \delta.\notag
\end{align}
We have used the volume comparison inequality in the last inequality and the constant $2^n$ we get from the inequality is absorbed in $c(n)$. 
Consider the set 
\begin{equation*}
    R_{a,b}\defeq\left\{x\in B_1(p): \sup_{r\in(0,2)}\dashint_{B_r(x)\cap B_1(p)}|\Ric(\nabla u_a,\nabla u_b)|d\vol_g>\sqrt{\delta}\right\}, 
\end{equation*}
for $a,b=1\ldots,n$. The weak $(1,1)$-estimate of maximal functions gives 
$$
\vol_g(R_{a,b})\le {c(n)}\sqrt\delta \vol_g(B_1(p)). 
$$ 
Choosing $\delta>0$ small, we see that 
\begin{align*}
    \vol_g\left(B_1(p)\setminus\bigcap_{a,b=1}^{n}(B_1(p)\setminus R_{a,b})\right)&\le \sum_{a,b=1}^n \vol_g( R_{a,b})\\
    &\le n^2c(n)\sqrt{\delta}\vol_g(B_1(p))\\
    &<\vol_g(B_1(p)).
\end{align*}
It follows by the above estimate and the Lebesgue differentiation Theorem that for every Lebesgue point $x\in \bigcap_{a,b=1}^{n}(B_1(p)\setminus R_{a,b})$, it holds 
\begin{equation}\label{eq:RicSmall}
    \Ric(\nabla u_a,\nabla u_b)(x)\le \sqrt{\delta},\ a,b=1,\ldots, n.
\end{equation}

Meanwhile, applying the transformation theorem as in \cite{BMS23}*{Proposition 4.3} (and its proof) to $\{u_a\}_{a=1}^n$, we get the following. There exists a Borel set $E$ so that $\vol_g(B_1(p)\setminus E)\le c(n)\sqrt \delta$. Denote $u\defeq (u_1,\ldots, u_n)$, for each $x\in E$, there exists a matrix $A_x$ with $|A^x-I|\le c(n)\sqrt{\delta}$ such that $u^x\defeq A^x\circ u$ satisfies
\begin{equation}\label{eq:ExactOrtho}
    \nabla u^x_a(x)\cdot \nabla u^x_b(x)=\delta_{a,b}.
\end{equation}
By the volume estimates for $R_{a,b}$ and $E$, there exists a Lebesgue point $x\in \bigcap_{a,b=1}^{n}(B_1(p)\setminus R_{a,b})\cap E$. This $x$ satisfies both \eqref{eq:RicSmall} and \eqref{eq:ExactOrtho}.

Take $v\in T_x M_i$ with $|v|=1$.
Incorporating the gradient estimates \ref{item:GradientEsti}, we can estimate the Ricci tensor evaluated at $x$. For the computation we denote the matrix elements as $A^x=(A^x_{a,b})_{a,b=1}^n$.
\begin{align*}
    \Ric(v,v)&= \sum_{a,b=1}^n (v\cdot \nabla u^x_a)(v\cdot \nabla u^x_b)\Ric(\nabla u^x_a,\nabla u^x_b) \\
    &\le \sum_{a,b=1}^n \sum_{c,d=1}^n  (A^x_{a,c}A^x_{b,d})^2(v\cdot \nabla u_c)(v\cdot \nabla u_d)\Ric(\nabla u_c,\nabla u_d)\\
    &\le c(n)(1+c(n)\sqrt\delta)^2\cdot \sqrt{\delta}\\
    &\le c(n)\sqrt{\delta}.
\end{align*}
 It follows that $\Sc(x)\le n  c(n)\sqrt{\delta}<2$ by choosing $\delta>0$ small enough, this is a contradiction. The same consideration also works for Theorem \ref{thm:gap}.
\end{remark}

\begin{proof}[Proof of Corollary \ref{cor:betti}]

We argue by contradiction. Suppose $b_1(X)\ge n-1$. 

When $X$ is compact, for large $i$, $M_i$ is also compact with uniform diameter upper bound. Thanks to the splitting theorem and \cite{Wang_RicciSemilocal}, the universal cover $\tilde X$ of $X$ splits off $\R^{n-1}$, and $\tilde X=\R^{n}$. Let $\tilde M_i$ be the universal cover of $M_i$ and $\tilde p_i$ be a lifting of $p_i$. By the Pan-Wang description of convergence of the sequence $\tilde M_i$ \cite{PanWang_universal}*{Theorem 1.5}, we obtain that there exists a Ricci limit space $(Y,y,G)$ as the equivariant GH limit of $(\tilde M_i, \tilde p_i, \pi_1(M_i))$, where $G$
is a closed subgroup of $\mr{Isom}(Y)$, and there exists a subgroup $H$ of $G$ such that $Y/H= \tilde X=\R^n$, so $Y=\R^n$ with standard Euclidean metric. In particular, a long torical band $[0,11]\times\mb T^{n-1}$ can be embedded into $\R^n$, so we get a contradiction as in the proof of \eqref{main:item1} of Theorem \ref{thm:limit_n-2}.

When $X$ is non-compact. It is proved by Ye \cite{Ye_bettiRCD}*{Theorem 2} (after Anderson \cite{andersontopology}) that $b_1(X)\le n-1$, so $b_1(X)=n-1$. Moreover, in this case $X$ is flat with the flat $\mb T^{n-1}$ as its soul. Here we have used that $X$ has no boundary to rule out $[0,\infty)\times \mb T^{n-1}$. This in turn implies that $X$ is isometric to $\R\times\mb T^{n-1}$ or to $\mb M^2\times \mb T^{n-2}$, where $\mb M^2$ is the open Mobius strip, see \cite{Ye_bettismooth}*{Proposition 4}. Note that when $X=\mb M^2\times \R^{n-2}$, $X$ has the flat $\R\times\mb T^{n-1}$ as its orientable double cover. We have shown in the proof of \eqref{main:item2} of Theorem \ref{thm:limit_n-2} that it suffices to consider the orientable double cover, so the problem reduces to showing that $\R\times\mb T^{n-1}$ cannot appear as a Ricci limit space for the sequence $M_n$. To this end, observe that $\R\times T^{n-1}$ contains an arbitrarily long toric band, the contradiction follows from the virtue of the proof of Theorem \ref{thm:limit_n-2}.  
\end{proof}

\begin{remark}
    
     However, the first Betti number is not lower semicontinuous w.r.t. the pGH convergence of non-compact manifolds. For example take a half cylinder $\mb S^1\times [0,\infty)$ and glue a hemisphere to the boundary $\mb S^1\times \{0\}$, we obtain a non-compact, simply connected $C^1$ manifold of non-negative curvature with uniform lower bound on the volume of the unit balls. Denote it by $M$. However, take $p_n=(\theta, n)$ for some fixed $\theta\in \mb S^1$, $(M,p_n)$ converges to the cylinder $\mb S^1\times \R$ as $n\to \infty$, which has Betti number $1$, for more details we refer to \cite{Wei_universalcover}.
\end{remark}

\vspace{2mm}
Then we prove the local volume gap Theorem \ref{thm:gap}. 

\begin{proof}
We argue by contradiction. 
Suppose that there exists a sequence $\{(M_i,g_i,p_i)\}_{i=1}^\infty$ such that
$$
\mr{vol}_{g_i}(B(p_i,1)) \rightarrow \omega_n.
$$
Then, by \cite{DPG17}*{Theorem 1.6} (for the isometry) and Theorem \ref{lem:top} (for the homeomorphism) there exists an unrelabled subsequence such that
$$
(M_i, g_i, p_i) \xrightarrow{\mr{pGH}} (M_\infty, \dist, p_\infty) 
$$
with the following properties:

\begin{enumerate}
    \item $B^i(p_i, \frac{99}{100}) \rightarrow B^{\R^n}(p_\infty, \frac{99}{100})$, here $B^{\R^n}(p_\infty, \frac{99}{100})$ is a $n$-dimensional Euclidean ball centered at $p_\infty$ with radius $\frac{99}{100}$ and $B^i$ is a geodesic ball in $M_i$;
    \item  There exists a sequence of bi-H\"older map, $\varphi_i: B^{\R^n}(p_\infty, \frac{1}{2}+ 2c) \rightarrow B^i(p_i,1) $ such that it is a homeomorphism onto its image and the image contains $B^{\mb{R}^n}(p_i,\frac12+c)$, and $\dist_g(\varphi_i(x), \varphi_i(y)) \rightarrow \dist(x,y)$ uniformly as $i \rightarrow \infty $ on $B^{\mb{R}^n}(p_\infty, \frac{1}{2})$ and $c$ is a small positive real number.
 \end{enumerate}

Now, we pick a cube $C_\infty \subset B^{\R^n}(p_\infty, \frac{1}{2})$ in the form
$$
C_\infty = \left[-\left(\frac{1}{2\sqrt{n}} + c\right), \frac{1}{2\sqrt{n}}+c\right]^n.
$$
Here $c$ is a small constant real number and $c$ may be different line by line.
Then, we push $C_\infty$ into $B^i(p_i,\frac{1}{2})$ by $\varphi_i$. Hence, we obtain that $C_i = \varphi_i(C)$ is a cube as well since $\varphi_i$ is a homeomorphism. Then for large $i$, we have the distance between any two opposite sides of $C_i$ can have a distance greater than $\frac{1}{\sqrt{n}} $ as $i \rightarrow \infty$; on the other hand, recall Theorem \ref{thm:cubeineq} that if $(D,g)$ is a $n$-cube with $Sc(g)\geq 4\pi^2 n(n-1)$, then the minimum among the distance of all $n$ opposite sides should be less than or equal to $\frac{1}{\sqrt{n}}$. This implies that the minimum among the distance of all $n$ opposite sides of $\varphi_i(C)$ will be less than or equal to $\frac{1}{\sqrt{n}}$. Hence, we reach a contradiction and conclude that there exists a constant $c_n$ such that
\begin{equation}\label{eq:volgap}
    \mr{vol}(B(p,1)) \leq c_n < \omega_n.
\end{equation}
\end{proof}

\begin{remark}
In fact, the pGH convergence $B_i(p_i,r)\to B^{\R^n}(p_\infty,r)$ holds for any $r\in(0,1)$, but for the closed balls $\bar B(p_i,1)\to \bar B^{\R^n}(p_\infty,1)$ cannot hold under the conditions in the above theorem. This can be seen from a unit ball in the cylinder $\mb S^1(1)\times\R$, which is not isometric to the unit ball in $\R^2$. This example is from \cite{DPG17}.
\end{remark}

Next, we sketch the proof of Corollary \ref{cor:gap_boundary}. It is derived by the same argument as in the previous proof so we omit the details. 

\begin{proof}
we use the volume rigidity for half-space in \cite{BNS20}*{Theorem 8.2} to find that if there exists $p_i\in \partial (M_i,g_i)$ such that $\vol_{g_i}(B(p_i,1))\to \frac12 \omega$, then $(M_i,g_i,p_i)\to (M_\infty,\dist, p_\infty)$, and $B(p_\infty,\frac{99}{100})$ is isometric to $B^{\R^n_+}(0,\frac{99}{100})$, a Euclidean (half) ball in upper half-space of $\R^n$ centered at the origin. Then we use the topological structure theorem \cite{BNS20}*{Theorem 1.2 (iii)} to find, for sufficiently large $i$, an GH approximation $\varphi_i: B^{R^n_+}(0,\frac{51}{100}) \to B^i(p_i,\frac{51}{100})$ whose image contains $B^i(p_i,\frac12+c)$ for some $c>0$ small and $\varphi_i$ is also a bi-H\"older homeomorphism onto its image. Finally, we embed a cube of size $\frac{1}{\sqrt{n}}+c$ in $B^{\R^n_+}(0, \frac{51}{100})$ to $B^i(p_i,\frac{51}{100})$ via $\varphi_i$, to deduce a contradiction to the scalar curvature bound. We have established the volume gap \eqref{eq:gap_bdry}. 
\end{proof}

Cai \cite{Cai_Largemanifold}*{Theorem 3} and Shen \cite{Shen_Largemanifold}*{p.11} obtained that once such a volume gap as in \eqref{eq:volgap} exists, it can be improved, but we can not follow up the proof completely. Their result would imply that in the same setting of Theorem \ref{thm:gap}, there exists a positive constant $c\defeq c(n)$ such that for any $q \in M$,
$$
\vol_g(B(q,r)) \leq c r^{n-1}.
$$

Instead, we prove Theorem \ref{thm:splitting}. 

\begin{proof}
We argue by contradiction. First we prove \eqref{eq:infvol}. Assume there exists a sequence of positive numbers $C_i\to \infty$ so that there exists a sequence $r_i$ with
\begin{equation}
    \inf_p \vol_g(B(p,r_i))>C_i r_i^{n-2}.
\end{equation}
In particular, for any $p\in M$, we have 
\begin{equation}\label{eq:volgrowth}
    \vol_g(B(p,r_i))> C_i r_i^{n-2}.
\end{equation}
Since $\vol_g(B(p,r_i))\le \omega_n r_i^n$ by volume comparison, we have that $r_i^2\ge (\omega_n)^{-1}C_i$ as $i\to \infty$, which implies that $r_i\to \infty$. Moreover, without loss of generality, we can assume $r_1>1$, then by volume comparison $\vol_g(B(p,1))\ge r_1^{-n}\vol_g(B(p,r_1))\ge {C_1}r_1^{-2}>0$, that is, $M$ is everywhere non-collapsed.

For any integer $k\in [1,n]$ we say a $\RCD(K,n)$ space $(X,\dist,\meas)$ splits off $\R^k$ at infinity, if there exists a sequence of points $x_i\in X$, $x_i\to \infty$, meaning for any fixed $x_0\in X$, $\dist(x_0,x_i)\to \infty$ as $i\to \infty$, such that $(X,\dist,\meas, x_i)$ pmGH converges and the limit splits off $\R^k$. 

Now we claim that $M$ splits off $\R^{n-1}$ at infinity. More precisely, there exists sequence of points $\{q_i\}_{i\ge 0}$ such that $(M,g,q_i)$ splits off $\R^{n-1}$ in the pGH limit. Then this is a contradiction to item \eqref{main:item1} of Theorem \ref{thm:limit_n-2}.

So it remains to prove the claim. Although the claim holds for $n\ge 3$, the case $n=3$ is extensively studied as pointed out in the introduction, we focus on $n\ge 4$. The proof is a successive splitting procedure as in the proof of \cite{Shen_Largemanifold}*{Theorem 1.3},  Since $(M,g)$ is non-compact, for any sequence $p_i\to \infty$ there is a ray $\sigma_i:[0,\infty)\to M$ such that $\sigma_i(0)=p_i$. Consider points $p^0_i\defeq\sigma(r_i)$. By almost splitting theorem, $(M,g,p^0_i)$ splits off an $\R$ factor, that is 
\begin{equation}
    (M,g,p^0_i)\xrightarrow{\rm{pGH}}  (X^0,\dist_0, p^0_\infty)=\left(\tilde X^0\times \R,\sqrt{\tilde \dist_0^2+\dist_{\R}^2},(\tilde p^0_\infty,0)\right).
\end{equation}
Meanwhile the volume measure $\vol_g$ converges to $\haus^n$ on $(X_0,\dist_0)$ in the following sense,
\begin{equation}\label{eq:volconv}
    \vol_g(B(p^0_i,r))\to \haus^n(B(p^0_\infty, r)), \quad \forall r\ge 0,
\end{equation}
and $\haus^n$ splits as $\haus^n=\haus^{n-1}\times \leb^1$. The volume growth condition \eqref{eq:volgrowth} and the volume convergence \eqref{eq:volconv} imply that 
\begin{equation}\label{eq:volesti}
   \haus^{n-1}(B(\tilde p^0_\infty, r_i))=\frac{\haus^n(B(p^0_\infty, r_i))}{r_i}\ge C_i r_i^{n-3}\to \infty \text{ as $i\to \infty$},
\end{equation}
which shows $(\tilde X_0,\tilde\dist_0)$ is non-compact hence has a ray at every point, we can thus find a sequence of points $ r^0_i\to \infty$ in $\tilde X_0$ so that $(\tilde X_0,\tilde \dist_0, r^0_i)$ splits off an $\R$ at infinity. Furthermore, this convergence can be realized by a sequence of points $p^1_i$ in $M$. More precisely, 
\begin{equation}
    (M,g, p^1_i)\xrightarrow{\rm{pGH}} (X^1,\dist_1, p^1_\infty)= \left(\tilde X^1\times \R^2,\sqrt{\tilde \dist_1^2+\dist_{\R^2}^2},(\tilde p^1_\infty,0)\right).
\end{equation}
Again, the measure splits as $\haus^{n-2}\times \leb^2$. We can apply this argument $(n-2)$-times and perform the volume estimates \eqref{eq:volesti} to find out that there exists a sequence of points $p^{n-3}_i\to \infty$ so that $(M,g, p^{n-3}_i)$ splits off $\R^{n-2}$ in the limit $(X^{n-3},\haus^n, p^{n-3}_\infty)= (\tilde X^{n-3}\times \R^{n-2},\haus^{2}\times \leb^{n-2}, (\tilde p^{n-3}_\infty,0))$ with volume estimates
\begin{equation}
    \haus^2 (B(p^{n-3}, r_i))\ge C_i\to \infty \text{ as $i\to \infty$}.
\end{equation}
So $\tilde X^{n-3}$ is non-compact, we can argue as above to find a sequence in $\tilde X^{n-3}$ so that $\tilde X^{n-3}$ splits off an $\R$ factor at infinity, and finally, take a sequence of points $q_i$ to realize this convergence from $M$, we then have found a sequence that splits $\R^{n-1}$, which is the claim.  

Now we show \eqref{eq:avr0}. Assume there exists $p$, so that
\begin{equation}
   \lim_{r\to\infty} \frac{\vol_g(B(p,r))}{r^n}=\alpha>0,
\end{equation}
then by volume comparison, we reach that $\vol_g(B(x,1))\ge \alpha >0$ for any $x\in M$. The proof above implies that there is a sequence $q_i\to \infty$ so that $(M,q_i)\to (\R^n,0)$, which contradicts \eqref{main:item1} of Theorem \ref{thm:limit_n-2}. This completes the proof.
\end{proof}

Finally, we prove Proposition \ref{minimalvolume} as follows:

\begin{proof} 
Suppose that $M^3$ is a simply connected three-dimensional 
Riemannian manifold and has only one end. Let $\Omega_0$ be a smooth, compact domain in $M$. For any $x \in M\setminus \Omega_0$, we define 
$$f_0(x)=\dist (x, \Omega_0).$$ 
Then, by \cite{Chodosh_Li_Soapbubble}*{Lemma 17}, there exists a two-sided, connected, closed surface $\Sigma_1 \subset M$ such that
\begin{enumerate}
    \item $\Sigma_1 \subset L(0, 4\pi):=\{x \in M: 0 \leq f_0(x) \leq 4\pi\}$;
     \item \label{diameter}$\text{diam}(\Sigma_1) \leq \sqrt{3}\pi.$
\end{enumerate}
Now, we pick $\Omega_0 = \{p\}$ for some $p \in M$. By \cite{AG}*{Proposition 4.1}, then there exists a constant $R= R(\Lambda)$ and a bounded domain $\Omega_u$ containing $p$ such that  for any $s> t \geq  R$, the unbounded connected set $\{ q \in M \setminus \Omega_u\} \bigcap \{q \in M , q \in  \dist^{-1}([t,s])\}$ is homeomorphic to $\Sigma_u \times [0,1]$ (See \cite{Cheeger_Criticalpoint}).

Let us first show that there exists a point $p \in M$ and constant $c_\Lambda\defeq c(\Lambda)$ such that
\[\vol_g(B(p,r)) \leq cr,\quad \forall r \geq 0.\]
Having assumed that $M$ has one end only, we obtain that for any $q \in M$, there exists a constant $R_1$ such that $\partial(B(q,R_1))$ is connected and divides $M$ into two components and $M\setminus B(q,R_1)$ is a connected, unbounded domain. Then, we consider $f(x)=\dist(x, \partial B(q,R_1))$ and then define
\[M_f= \max\{f(x): x \in B(q,R_1)\}.\]

Now, we take $p_0 \in B(q,R_1)$ such that $f(p_0) = M_f$. i.e. $p_0$ is the ``innermost'' point in $M$ geometrically. 

\textbf{Claim}: For $p_0$, there exists a constant $c>0$, such that
\begin{equation} \label{volumeatp0}
   \vol_g(B(p_0,r)) \leq c r,\quad  \forall r \geq 0. 
\end{equation}

Foremost, let's do the following preparation:  there exists a constant $R_\Lambda\defeq R(\Lambda)$ such that $\dist(p_0,x)$ has no any critical points in the unbounded connected of $M\setminus B(p_0,R_\Lambda)$ and we denote by $\Omega_u$ the unbounded connected of $M\setminus B(p_0,R_\Lambda)$ as follows. We claim that $\partial \Omega_u$ is connected. Suppose not, for any $r >R_\Lambda$, we have
\[ L(R_\Lambda, r) = \{x \in \Omega_u: R_\Lambda \leq f(x) \leq r\}. \]
has more than one connected components. Since $f(x) = \dist(p_0, x)$ has no critical points on $M\setminus B(p_0,R_\Lambda)$, it implies that that $M$ has more than one end, which contradicts with the assumption that $M$ has one end. Therefore, $ \partial \Omega_u$ is connected.

Moreover, for any $p_1 \in M\setminus \Omega_u$, we claim
\begin{equation}\label{distance}
    \dist(p_0,p_1) \leq 2R_\Lambda +30\pi.
\end{equation}
Let us prove the inequality (\ref{distance}). Let $\gamma_i$ be a minimizing geodesic segment such that for $i=0,1$,
\begin{itemize}
    \item $\gamma_i(0)=p_i, \gamma_i (d(p_i, \partial B(p,R))) \in \partial (B(p,R))$;
    \item $\gamma_i$ intersects with $\partial \Omega_u$ at $P_i$.
\end{itemize}
Then, by triangle inequality, we obtain
$$
\begin{array}{rcl}
   \dist(P_1, \partial B(p,R))& \geq&  (\dist(p, \partial B(p,R))-R_\Lambda)-20\pi.
\end{array}
$$
Since $\dist(p_0, \partial B(p,R)) \geq d(p_1, \partial B(p,R))$ by the choice of $p_0$ above, by triangle inequality again, we arrive at
$$
\dist(p_0,p_1) \leq \dist(p_0,P_0) + \dist (P_0,P_1) + \dist(P_1,p_1) +20\pi  \leq 2R_\Lambda +30\pi.
$$
Hence, we finish the proof of the inequality (\ref{distance}), and then set up $\bar{R}_\Lambda=2R_\Lambda +30\pi$. Hence, $B(p_0, \bar{R}_\Lambda)$
satisfies the properties as follows:
\begin{itemize}
    \item $\partial B(p_0, \bar{R}_\Lambda)$ is a two-dimensional topological sphere;
    \item $M\setminus B(p_0,\bar{R}_\Lambda)$ is connected, unbounded and does not contain any critical points of $\dist(p_0,x)$ in $ M\setminus B(p_0,\bar{R}_\Lambda)$.
\end{itemize}

\bigskip

Secondly, we go back the proof the volume growth estimate (\ref{volumeatp0}). We divide the proof into two cases as follows:

\begin{itemize}

\item \textbf{Case 1:} If $r \geq \bar{R}_\Lambda + 4\pi$ and $(n-1) < \frac{r-\bar{R}_\Lambda}{4\pi} \leq n$ for some positive $n \in \mathbb{Z}$, then 
\[
\vol_g(B(p_0, r)\setminus B(p_0,\bar{R}_\Lambda)) \leq \sum_{i=1}^n \vol_g (L(\bar{R}_\Lambda+4\pi (i-1)), \bar{R}_\Lambda+4\pi i)).
\]
Note that the volume comparison implies that
$$
\vol_g(B(p_0, r)\setminus B(p_0,\bar{R}_\Lambda)) \leq n \omega_3(12\pi)^3 \leq 2\omega_3(n-1)(12\pi)^3 \leq \frac{2\omega_3(12\pi)^3}{4\pi}(r-\bar{R}_\Lambda).
$$
Hence, for all $r  \geq \bar{R}_\Lambda + 4\pi $, we have
\begin{equation}\label{bigr}
    \vol_g(B(p_0, r)\setminus B(p_0,\bar{R}_\Lambda)) \leq \frac{12^3 \omega_3 \pi^2 }{2}(r-\bar{R}_\Lambda);
\end{equation}

\item \textbf{Case 2:} If $r \leq \bar{R}_\Lambda + 4\pi$, then we have by volume comparison theorem, 
\begin{equation}\label{smallr}
    \vol_g(B(p_0,r)) \leq (\omega_3r^2) r,
\end{equation}
and $\omega_3 r^2 \leq \omega_3(\bar{R}_\Lambda + 4\pi)^2$.

\end{itemize}

Hence, together with (\ref{bigr}) and (\ref{smallr}), there exists a constant $c_\Lambda$ such that
$$\vol_g(B(p_0,r)) \leq c_\Lambda r, \ \ r \geq 0.$$
Here, $c_\Lambda= \max\{2\omega_3(\bar{R}_\Lambda + 4\pi)^2, \frac{12^3 \omega_3 \pi^2 }{2}\}$.

\bigskip

Moreover, let us prove that there exists a constant $c=c_\Lambda$ such that for any $P \in M$ such that
\begin{equation}
{\vol_g(B(P,r))} \leq c_\Lambda r.
\end{equation}
Let us divide the proof into two cases: 

$\textbf{(A):}$ $\dist(P,p_0) \leq 2\bar{R}_\Lambda$:
\begin{itemize}
    \item If $r \leq 4\bar{R}_\Lambda$, by volume comparison, we have
    $$\vol_g(B(P,r)) \leq (4\bar{R}_\Lambda)^2 \omega_3 r.$$
    \item If $r \geq 4\bar{R}_\Lambda$, we have by the same argument above
    $$\vol_g(B(P,r)) \leq \vol_g(B(p_0,\bar{R}_\Lambda))+ \vol_g(B(P,r)\setminus B(p_0,\bar{R}_\Lambda)).$$
    Since $B(P,r)\setminus B(p_0,\bar{R}_\Lambda) \subset M\setminus B(p_0,\bar{R}_\Lambda)$, Hence
    \begin{equation}\label{outpart}
        \vol_g(B(P,r)\setminus B(p_0,\bar{R}_\Lambda)) \leq c_\Lambda r.
    \end{equation}
    Together with two cases in \textbf{(A)}, we reach that there exists a constant $c_\Lambda$ such that for any $r \geq 0$,
    $$\vol_g(B(P,r)) \leq c_\Lambda r.$$

\end{itemize}

$\textbf{(B):}$ $\dist(P,p_0) \geq 2\bar{R}_\Lambda$:
\begin{itemize}
    \item  If $r \leq 2\bar{R}_\Lambda$, by volume comparison theorem, 
    $$\vol_g(B(P,r)) \leq c_\Lambda r;$$
    \item if $r \geq 2\bar{R}_\Lambda$, we have
    $$\vol_g(B(P,r)) \leq \vol_g(B(p_0,\bar{R}_\Lambda)) + \vol_g(B(P,r)\setminus B(p_0, \bar{R}_\Lambda)) $$
    By the same argument in (\ref{outpart}), we have
    $$\vol_g(B(P,r)) \leq c_\Lambda r.$$
\end{itemize}
Hence, combining $(A)$ and $(B)$,  we obtain that there exists a constant $c_\Lambda$ such that, for any $p$ and $r$,
$$\vol_g(B(p,r)) \leq c_\Lambda r.$$

\bigskip
Finally, let us prove the volume asymptotic behavior:
\begin{equation}
   \limsup_{r \rightarrow \infty } \frac{\vol_g(B(p,r))}{r} \leq  \limsup_{r \rightarrow \infty } \frac{\vol_g(p_0, \bar{R}_\Lambda) + \vol_g(B(p, r)\setminus B(p_0,\bar{R}_\Lambda))}{r} < c, \ \text{for some} \ c>0.
\end{equation}

Here $c$ is a universal constant since $ \vol_g(B(p, r)\setminus B(p_0,\bar{R}_\Lambda)) \subset M\setminus B(p_0,\bar{R}_\Lambda)$ and then
$\vol_g(B(p, r)\setminus B(p_0,\bar{R}_\Lambda)) \leq c r$. We have completed the proof.
\end{proof}


\textbf{Conflict of interests declaration.} The authors declare that there is no conflict of interests.

\textbf{Data availability Statement.} There are no data associated to this article.

\bibliographystyle{amsalpha}
\bibliography{scalarcurvature}

\end{document}